\documentclass[11pt]{amsart}
\usepackage{amssymb, amsmath}

\newtheorem{theorem}{Theorem}[section]
\newtheorem{lemma}[theorem]{Lemma}
\newtheorem{assertion}[theorem]{Assertion}
\newtheorem{definition}[theorem]{Definition}
\newtheorem{proposition}[theorem]{Proposition}
\newtheorem{corollary}[theorem]{Corollary}
\theoremstyle{remark}
\newtheorem{remark}[theorem]{Remark}
\numberwithin{equation}{section}
\allowdisplaybreaks[3]

\begin{document}
\title[Abel averages and holomorphically pseudo-contractive maps]{Abel
averages and holomorphically pseudo-contractive maps in Banach
spaces}
\author[F. Bracci]{Filippo Bracci$^\dag$}
\address{F. Bracci: Dipartimento Di Matematica\\
Universit\`{a} di Roma \textquotedblleft Tor Vergata\textquotedblright\ \\
Via Della Ricerca Scientifica 1, 00133 \\
Roma, Italy}
\email{fbracci@mat.uniroma2.it}
\author[Y. Kozitsky]{Yuri Kozitsky$^\ddag$}
\address{Y. Kozitsky: Instytut Matematyki, Uniwersytet Marii Curie-Sk{\l}odowskiej, Plac Marii Curie-Sk{\l}odowskiej 1, 20-031 Lublin, Poland}
\email{jkozi@hektor.umcs.lublin.pl}

\author[D. Shoikhet]{David Shoikhet$^{\ddag}$}

\address{D. Shoikhet: Department of Mathematics, ORT Braude College, 21982 Karmiel, Israel}
\email{davs@braude.ac.il}

\thanks{$\dag$ Supported by the ERC grant ``HEVO - Holomorphic Evolution Equations'' n. 277691.}
\thanks{$\ddag$ Supported by the European Commission under the project
STREVCOMS PIRSES-2013-612669.}
\begin{abstract}
A class of maps in a complex Banach space is studied, which includes
both unbounded linear operators and nonlinear holomorphic maps. The
defining property, which we call {\sl pseudo-contractivity}, is
introduced by means of the Abel averages of such maps. We show that
the studied maps are dissipative in the spirit of the classical
Lumer-Phillips theorem. For pseudo-contractive holomorphic maps, we
establish the power convergence of the Abel averages to holomorphic
retractions.
\end{abstract}

\maketitle

\section{Introduction and the results}\label{S1}

\subsection{Preliminary and notations}
In this paper,  $(X,\|\cdot\|)$ will stand for a complex Banach
space. The best studied class of self-maps of $X$ is that of bounded
linear operators $T:X \to X$, which we denote by $\mathcal{B}(X)$.
For $T\in \mathcal{B}(X)$, the \textit{Abel average} is defined as
\begin{equation}  \label{i1}
A_\alpha = (1-\alpha) \sum_{n=1}^\infty \alpha^n T^n,
\end{equation}
where $\alpha \in (0, 1)$ is such that the series in (\ref{i1})
converges in the operator norm topology. The study of Abel averages
of bounded linear operators goes back to, at least, E. Hille
\cite{Hil} and W. F. Eberlein \cite{Eber}. Two natural extensions of
$\mathcal{B}(X)$ are the class of unbounded linear operators
$T:\mathcal{D}(T)\subset X \to X$, and the class of nonlinear
holomorphic maps $h:V\to X$ with various choices of the domains $V$.
Nonlinear holomorphic maps play an important role in complex
analysis and in theory of   dynamical systems, see, {\sl e.g.},
\cite{BES,H,H1,HRS,ReichS,Ves}. Note that, for infinite dimensional
Banach spaces, these two classes are disjoint. Separately they are
studied extensively. However, in many applications one encounters
maps that are neither linear nor continuous. A typical example is
the map $T + g$ where $T$ is a linear unbounded operator and
$g$ is nonlinear and holomorphic. Such maps appear, in particular,
in evolution equations of reaction-diffusion type, see, {\sl e.g.},
\cite{B} and subsection \ref{Essec} below.

The aim of the present paper is:
\begin{enumerate}
\item to introduce a class of maps -- {\sl  $\omega$-pseudo-contractive maps}
-- which includes both unbounded  linear operators
and nonlinear holomorphic maps, and to study this class by means of
properly defined Abel averages;
\item to study, cf. \cite{BES,H1,HRS}, the Abel averages of nonlinear holomorphic
maps in the spirit of \cite{KSZ}, where a number of properties of
the Abel averages of unbounded linear operators were obtained.
\end{enumerate}

For a linear operator $T$, we write $\mathcal{D}(T)$, $\mathrm{Ker}
(T)$, and $\mathrm{Im} (T)$ for denoting the domain, kernel, and
range of $T$, respectively. By $I$ we denote the identity operator
$Ix = x$, $x\in X$. Given $r>0$ and $x\in X$, we set $B_r(x) =
\{y\in X: \|y-x\|<r\}$. For brevity, we denote $B_r(0)$ by $B_r$,
and $B_1(0)$ by $B$. An open and connected subset $V\subset X$ is
called a domain. A map $h:V \to W\subseteq X$ is called holomorphic
if it admits the Fr{\'e}chet derivative $h'(x)$  at each $x\in V$.
For $W\subseteq X$, by $\mathrm{Hol}(V, W)$ we denote the set of all
holomorphic maps from a domain $V$ to $W$, and write
$\mathrm{Hol}(V)$ for $\mathrm{Hol} (V,V)$.

For $T\in \mathcal{B}(X)$, uniform ergodic theorems for Abel and
Ces{\'a}ro averages were established by M. Lin \cite{Lin0, Lin}.
 The following assertion can be
deduced from the corresponding classical results of \cite{Lin0,Lin},
see \cite[Assertion 1.3]{KSZ}.
\begin{assertion}
\label{ias1}
 Let $T\in \mathcal{B}(X)$ be  such that $\|T^n/n\|\to 0$ as $n\to +\infty$.
Then, for each $\alpha \in (0,1)$, the series in (\ref{i1})
converges in the norm topology. Moreover, the following statements
are equivalent: \vskip.1cm
\begin{itemize}
\item[{\it (i)}] $\mathrm{Im}(I-T)$ is a closed subset of $X$.
\item[{\it (ii)}]   For some $\alpha \in (0,1)$, the sequence
$\{A^n_\alpha\}_{n \in
\mathbb{N}}$ converges in the operator norm topology as $n\to +\infty$.
\item[{\it (iii)}] For each $\alpha \in (0,1)$, the sequence
$\{A^n_\alpha\}_{n \in
\mathbb{N}}$ converges in the operator norm topology as $n\to +\infty$.
\end{itemize}
\vskip.1cm \noindent The limit in (ii) and (iii) is given by the
projection of $X$ onto $\mathrm{Ker}(I-T)$ along $\mathrm{Im}(I-T)$.
\end{assertion}
For an unbounded operator $T:\mathcal{D}(T)\subset X \to X$, the Abel average
is defined as
\begin{equation}  \label{i3}
A_\alpha = (1-\alpha) [I- \alpha T]^{-1},
\end{equation}
where $\alpha\in(0,1)$ is such that $A_\alpha\in \mathcal{B}(X)$. A
further extension of this kind consists in defining Abel averages
for nonlinear holomorphic maps in $X$. The Abel averages of these maps
are also nonlinear holomorphic maps. They provide effective tools
for studying nonlinear holomorphic maps, in particular nonlinear
holomorphic semigroups and their fixed point sets \cite{ReichS}.

Of course, the condition $\Vert T^{n}/n\Vert \rightarrow 0$ in
Assertion \ref{ias1} is not applicable to unbounded operators.
Moreover, the condition is evidently far from necessary for the
corresponding convergence to hold; it can be replaced by, {\sl e.g.}, the
dissipativity condition used in the classical Lumer-Phillips
theorem, see \cite[p. 30]{Bonsal}. It turns out that the only
essential property  of $T$ required to guarantee the convergence in
Assertion \ref{ias1} is that the spectrum $\sigma (T)$, except
possibly for the point $\zeta =1$, lies in the half-plane
\begin{equation*}
\Pi =\{\zeta \in \mathbb{C}:\mathrm{Re}\zeta <1\}.  
\end{equation*}
Note that the conditions $\Vert T^{n}/n\Vert \rightarrow 0$ and
\textit{(i)} in Assertion \ref{ias1} imply that
\begin{equation}
\mathrm{Ker}(I-T)\oplus \mathrm{Im}(I-T)=X.  \label{i4}
\end{equation}%
Set $\rho (T)=\mathbb{C}\setminus \sigma (T)$. In \cite[Theorem
2.1]{KSZ}, a necessary and sufficient condition for the power
convergence of the Abel averages (\ref{i3})  was presented in
the following form.
\begin{assertion}
\label{ias2} Let $T$ be a densely defined linear operator such that
$(1,+\infty )\subset \rho (T)$. Let also $A_{\alpha }$ be its Abel
average (\ref{i3}). Then the following statements are equivalent.
\vskip.1cm
\begin{itemize}
\item[{\it (i)}] For each $\alpha \in (0,1)$, $A_{\alpha }\in
\mathcal{B}(X)$ and the  sequence $\{A_{\alpha }^{n}\}_{n\in \mathbb{N}}$
converges in the operator
norm topology.
\item[{\it(ii)}] $\sigma (T)\subset \Pi \cup \left\{ 1\right\} $ and
(\ref{i4}) holds.
\end{itemize}
\vskip.1cm \noindent For each $\alpha \in (0,1)$, the limit in (i) is given
by the projection of $X$ onto $\mathrm{Ker}(I-T)$ along $\mathrm{Im}(I-T)$.
\end{assertion}
Let $X^*$ be the dual space of $X$.  For $x^*\in X^*$ and $x\in X$,
by $\langle x, x^* \rangle$ we denote the value of the functional
$x^*$ at $x$. Then we set $J(x)=\{x^\ast\in X^\ast: \|x\|^2=\langle
x, x^\ast\rangle=\|x^\ast\|^2\}$. A linear operator $T$ is called
{\it dissipative}, see \cite[Definition 3.13 and Proposition
3.13]{EN}, if for each $x\in \mathcal{D}(T)$, there exists $x^* \in
J(x)$ such that the following holds
\begin{equation}
  \label{EN}
 {\rm Re}\langle Tx , x^* \rangle \leq 0.
\end{equation}
Similar definitions were used also for nonlinear holomorphic maps,
see \cite[Definition 2]{HRS}.
Note that a densely defined closed linear operator $T$ is {\sl dissipative} if
and only if, for all $\alpha \geq 0$,
\begin{equation}
  \label{i7}
 \|A_\alpha\| \leq 1.
\end{equation}
In view of their numerous
applications, dissipative maps constitute an important class of maps
in complex Banach spaces. In this work, we introduce the
following notion.
\begin{definition}
\label{disdf} Given $\omega \in \mathbb{R}$ and a dense subset
$D\subset B$, a map $h : D \to X$ is called $\omega$-dissipative if,
for some $\varepsilon >0$ and $\varsigma \in \mathbb{R}$, and for
each $x\in D$ such that $1 - \varepsilon < \|x\| < 1$, and all $x^*
\in J(x)$, the following inequality holds
\begin{equation}  \label{3b}
\mathrm{Re} \langle h(x), x^* \rangle \leq \omega \|x\|^2 + \varsigma (1 -
\|x\|^2).
\end{equation}
\end{definition}
Obviously, a $0$-dissipative linear operator is dissipative in the
sense of (\ref{EN}). Below we study the connection between
$\omega$-dissipativity of maps and the existence and contractiveness
of their Abel averages, which we define as follows. For a map $h : D
\subset B \to X$, and for real $\omega$ and $\alpha\neq 1/\omega$,
we set, cf. (\ref{i3}),
\begin{equation}  \label{4}
\Phi_\alpha = (I - \alpha h)^{-1} \circ [(1 - \alpha \omega ) I], \qquad
\Phi_0 = I.
\end{equation}
\begin{definition}
\label{gdf} A map $h:D \to X$ is called $\omega$-pseudo-contractive
if there exists $\delta>0$ such that for each $\alpha \in (0,
\delta)$, the Abel average $\Phi_\alpha$ defined in (\ref{4}) is in
$\mathrm{Hol}(B, D)$.
\end{definition}
Note that in Definition \ref{gdf} we do not require $D$ be dense in
$B$.
\begin{remark}
  \label{szymon}
The use of pseudo-contractive nonlinear maps goes back to papers by
 F. E. Browder \cite{Browder}
and W. A. Kirk \cite{Kirk}. In those works, $h$ is set to be
pseudo-contractive if $I-h$ is  accretive, and hence $h-I$ is
dissipative. A detailed study of accretive nonlinear maps can be
found in S. Reich's paper \cite{Reich}, cf. Theorem 3.1 on page 32.
As we will see below in Theorem \ref{2tm}, our term
$\omega$-pseudo-contractiveness with $\omega=1$ is parallel to the
pseudo-contractiveness in the works just mentioned.
\end{remark}

\begin{definition}
  \label{weakdf}
A map $h:D \to X$ is said to be closed in the weak topology if, for each
norm-convergent sequence $\{x_n \}_{n\in \mathbb{N}}\subset D$ such
that $\lim_{n\to \infty} x_n =: x \in D$ and such that the sequence
$\{h(x_n) \}_{n\in \mathbb{N}}$ weakly converges to  $y\in
X$, it follows that $h(x) = y$.
\end{definition}
Let $V$ be a domain containing $0$. Following L. Harris \cite{H}, we
say that the \textsl{spectrum} $\sigma(h)$ of $h\in
\mathrm{Hol}(V,X)$ consists of those $\lambda \in \mathbb{C}$, for
which there exist no open sets $U$, $W $ such that $0\in U\subset
V$, $W\subset X$, and $\lambda I - h$ is a holomorphic
diffeomorphism of $U$ onto $W$. It is known, see \cite{H}, that
\begin{equation*}  
\sigma (h) = \sigma (h^{\prime }(0)).
\end{equation*}
Given $h \in \mathrm{Hol}(B,X)$, suppose that $h$ can continuously be  extended
the boundary $\partial B$. Then the \textsl{numerical range} of $h$ is, cf. \cite{H,H1},
\begin{equation*} 
N (h) = \{ \langle h_s(x) , x^*\rangle : s\in (0,1), \ \ x \in
\partial B, \ \ x^* \in J(x)\}.
\end{equation*}
For $h \in \mathrm{Hol}(B,X)$ and
and $s\in (0,1)$, we let $h_s(x) = h(sx)$, $x\in B$. Then the  \textsl{numerical radius} of $h$ is
\begin{equation}  \label{H2}
L(h) = \limsup _{s\to 1^{-}} \mathrm{Re} N(h_s).
\end{equation}
If $h$ has a uniformly continuous extension to $\bar{B}$ then
$$L(h)= \lim_{t\to 0^{+}} \frac{\|I + t h\| - 1}{t},$$
by \cite[Theorem 2]{H} (see also the related
discussion in \cite[Section 2]{HRS}).

Each $h\in {\rm Hol}(V,X)$ has the following property: for every
$x\in V$, there exists $r>0$ such that $h$ is bounded on
$\bar{B}_r(x)$. From this one readily gets that $h$ is bounded on
each compact subset of $V$. However, $h$ need not be bounded on {\sl
each} bounded subset of $V$. Following \cite{BES,HRS} we say that a
holomorphic map $h:B \to X$ has \textit{unit radius of boundedness}
if, for each $r\in (0,1)$,
\begin{equation*}  
\sup_{x\in \bar{B}_r}\|h(x)\| := C_h(r) < \infty.
\end{equation*}
For nonlinear maps $g:U \to X$, $U \subseteq X$,  we use the following notations
\begin{eqnarray*} 
\mathrm{Null} (g) & = & \{ x\in U: g(x) = 0\}, \\[.2cm]
\mathrm{Im} (g) & = & \{ y\in X: \exists x \in U \ g(x) = y\}.  \notag
\end{eqnarray*}

\subsection{The results}
Let $\Delta:=\{\zeta\in \mathbb C: |\zeta|<1\}$ and
$\bar{\Delta}:=\{\zeta\in \mathbb C: |\zeta|\leq 1\}$. Recall that a
set $C\subset X$ is called {\sl balanced} if $\zeta \in
\bar{\Delta}$ for each $x\in C$ and $\zeta x \in C$.
\begin{theorem}
\label{g1tm} Let $D\subset B$ be dense in $\mathit{B}$ and balanced.
Let also $h:D\rightarrow X$ be an $\omega $-pseudo-contractive map,
closed in the weak topology. Assume that  either $X$ is reflexive,
or $X$ is weakly sequentially complete and $h$ is bounded on every
$K(x):= \{\zeta x:   \zeta \in \Delta\}$,  $x\in D$. Then $h$ is
$\omega$-dissipative.
\end{theorem}
Note that densely defined
closed linear operators are closed in the weak topology and bounded
on the sets $K(x)$.  Indeed, for a linear operator $T$, the graph $\Gamma(T)\subset X\times X$ is a convex set, which is
weakly closed whenever $T$ is closed in the usual sense. Moreover,
$TK(x) = K(Tx)$, which yields that $T$ is bounded on
compact subsets of $K(x)$. Set
\begin{eqnarray}  \label{5}
& & Q_\omega = [0, 1/\omega), \ \ \omega>0; \quad Q_0 = [0,
+\infty),
\\[.2cm]
& & Q_\omega = (-\infty , 1/\omega) \cup [0, +\infty), \ \ \omega<0.
\notag
\end{eqnarray}
\begin{theorem}
\label{2tm} Given $h \in {\rm Hol}(B, X)$  and $\omega\in
\mathbb{R}$, the following statements are equivalent:
\begin{itemize}
\item[{\it(i)}] $h$ is $\omega$-dissipative.
\item[{\it (ii)}] $L(h) \leq \omega$.
\item[{\it (iii)}] $h$ has unit radius of boundedness and it is $\omega$%
-pseudo-contractive. Moreover, for all $x\in B$,
\begin{equation}  \label{H4}
\lim_{\alpha \to 0^{+}} \Phi_\alpha (x) = x, \qquad \lim_{\alpha \to 0^{+}}%
\frac{\Phi_\alpha (x) - x}{\alpha} = h(x) - \omega x,
\end{equation}
where the convergence is pointwise in the norm topology of $X$.
\item[{\it (iv)}] $h$ has unit radius of boundedness and $\Phi_\alpha\in
\mathrm{Hol}(B)$, for each $\alpha \in Q_\omega$.
\end{itemize}
\end{theorem}
As an immediate consequence of Theorem \ref{g1tm} and Theorem
\ref{2tm} we have the following:
\begin{corollary}
\label{Coco}
Let $X$ be a weakly sequentially complete Banach spaces, $\omega\in \mathbb R$ and
$h\in {\rm Hol}(B,X)$. Then $h$ is $\omega$-dissipative if and only if  it is
$\omega$-pseudo-contractive.
\end{corollary}
Next, we study the convergence of nonlinear Abel averages in the
holomorphic category. As usually, for $h\in \mathrm{Hol}(B)$ and
$n\in \mathbb{N}_0$, by $h^{n}$  we denote  the $n$-th iterate of
$h$, {\sl i.e.}, $h^{n}:=h^{n-1}\circ h$, $h^{0}=I$.
\begin{theorem}
\label{3tm} Let $h \in \mathrm{Hol} (B,X)$ be such that $h(0) =0$. Suppose
that, for some $\omega \in \mathbb{R}$, the following holds
\begin{equation}  \label{13}
\mathrm{Ker} (\omega I - h^{\prime }(0)) \oplus \mathrm{Im} (\omega I -
h^{\prime }(0))= X.
\end{equation}
Assume also that $\Phi_\alpha\in\mathrm{Hol} (B)$ for a certain
$\alpha\in \mathbb{R}\setminus \{0\}$ such that $\alpha\omega \neq
1$. Then the sequence $\{ \Phi_\alpha^n\}_{n\in \mathbb{N}}$
converges in the norm topology of $X$, uniformly on closed subsets
of $B$, if and only if the following holds
\begin{equation}  \label{14}
\sigma (h^{\prime }(0))\subset \Omega(\alpha , \omega):= \{\omega\}
\cup \left\{ \zeta \in \mathbb{C}: \left\vert\zeta - \frac{1}{
\alpha} \right\vert
> \left\vert \omega - \frac{1}{\alpha}  \right\vert\right\}.
\end{equation}
Moreover, the limit of $\{\Phi^n_\alpha\}_{n\in \mathbb{N}}$ is a
holomorphic retraction
\begin{equation*}
\phi_\alpha: B \to \mathrm{Null}(\omega I - h):=\{x \in B: h(x)=\omega x\}.
\end{equation*}
\end{theorem}
Contrary to the case of linear operators, the retractions
$\phi_\alpha$ may depend on $\alpha$, see Section \ref{dep} below.

\section{Proof of the theorems}

\subsection{Dissipative and pseudo-contractive maps}

\label{dis}

\begin{lemma}
\label{D1lm} Let $h:D \to X$ be $\omega$-pseudo-contractive. Then $D$ is
dense in $B$ if and only if, for each $x\in B$,
\begin{equation}  \label{gG}
\lim_{\alpha \to 0^{+}} \Phi_\alpha (x) = \Phi_0 (x)= x,
\end{equation}
where the convergence is in the norm topology of $X$.
\end{lemma}

\begin{proof}
If (\ref{gG}) holds for all $x\in B$, then $D$ is dense since $\Phi_\alpha
(x) \in D$ for all $\alpha \in (0,\delta)$ and $x\in B$.

Let us prove the converse. First we prove that (\ref{gG}) holds for all $%
x\in D$. For $\alpha \in (0,\delta)$, we define
\begin{equation}  \label{g0}
y_\alpha = \frac{x - \alpha h(x)}{1- \alpha \omega }, \qquad x\in D.
\end{equation}
Clearly, $y_\alpha \in B$ for small enough $\alpha$, and hence we can
compute $\Phi_\alpha (y_\alpha)$ and get $\Phi_\alpha (y_\alpha)=x$. Let $%
\rho$ be the hyperbolic metric on $B$.\footnote{$\rho$ is a
pseudometric on $B$ satisfying the Schwarz-Pick inequality. All such
pseudometics coincide; they are metrics on $B$, see, {\sl e.g.}, \cite[page 90]%
{GoebelR} or \cite[Section 3.6, page 97]{ReichS}.} Then
\begin{equation}  \label{contract-hyp}
\rho (\Phi_\alpha (x), x) = \rho (\Phi_\alpha (x), \Phi_\alpha (y_\alpha))
\leq \rho (x,y_\alpha),
\end{equation}
since $\Phi_\alpha \in \mathrm{Hol}(B)$. It is known \cite[Theorem 3.7, page
99]{ReichS} that $\rho$ is locally equivalent to the norm metric of $B$.
Moreover, for each $x\in B$ and $l\in (0,1)$, there exist positive $c_1
(x,l) $ and $c_2 (x,l)$ such that, for all $y \in \bar{B}_l$,
\begin{equation}  \label{d3}
c_1 (x, l) \| x-y\| \leq \rho(x,y) \leq c_2 (x,l) \|x-y\|.
\end{equation}
Now we choose some $l\in (\|x\|, 1)$ and $\alpha_l>0$ such that $%
y_{\alpha_1} $ in (\ref{g0}) lies in $\bar{B}_l$ for $\alpha < \alpha_l$.
Then by (\ref{d3}) and \eqref{contract-hyp},
\begin{equation}  \label{d4}
\rho (\Phi_\alpha (x), x) \leq \frac{c_2 (x,l) \alpha}{ 1 - \alpha \omega}
\| \omega x - h(x)\|,
\end{equation}
which implies (\ref{gG}) for this $x$.

Since $D$ is dense in $B$, the general case of $x\in B$ can be
handled by the triangle inequality and the result just proven.
\end{proof}
\noindent As a direct consequence of (\ref{gG}) we have the
following result.
\begin{corollary}
\label{G1cor} Let $D\subset B$ be dense and $h:D\subset B \to X$ be
$\omega$-pseudo-contractive. Then, for each $x\in B$ and $r\in
(\|x\|,1)$, there exists $\delta_r < \delta$ such that $\Phi_\alpha
(x) \in \bar{B}_r$ whenever $\alpha \in [0, \delta_r]$.
\end{corollary}
Next, we need the following version of \cite[Lemma 4]{HRS}.
\begin{proposition}
\label{HRSpn} Given $\vartheta >0$, let $f:\Delta \times
[0,\vartheta) \to \Delta$ be holomorphic in the first variable for
each fixed $t \in [0,\vartheta)$, and right-differentiable at $0$ in
the second variable for each fixed $\zeta \in \Delta$. Suppose also
that $f(\zeta, 0) = c \zeta$ for all $\zeta \in \Delta$ and some
$c\in (0,1]$. Then, for each $\zeta \in \Delta$,
\begin{equation}  \label{H}
\mathrm{Re} \bar{\zeta} f^{\prime }_t (\zeta , 0) \leq (1- c^2|\zeta|^2)
\mathrm{Re} \bar{\zeta} f^{\prime }_t (0 , 0),
\end{equation}
where $f^{\prime }_t (\zeta, t) = \partial f(\zeta, t)/ \partial t$.
\end{proposition}
In the sequel, by ${\hbox{{\it w}-$\lim$}}$ we denote  the limit in the weak topology of $X$.
\begin{lemma}
\label{H2lm} Let $D\subset B$ be balanced and dense and let $h:D \to X$ be $%
\omega$-pseudo-contractive for some $\omega \in \mathbb{R}$. Suppose
also that, for each $x\in D$, the following holds
\begin{equation}  \label{d5A}
\underset{\alpha \to 0^{+}}{\hbox{w-$\lim$}} \frac{ \Phi_\alpha (x) - x}{%
\alpha} = h(x) - \omega x.
\end{equation}
Then $h$ is $\omega$-dissipative.
\end{lemma}
\begin{proof}
For $t>0$, we set $\alpha_t = t/(1+\omega t)$, and let
$\vartheta_\delta>0$ be such that $\alpha_t \in [0, \delta)$ for
$t\in [0, \vartheta_\delta)$. Fix $x \in D\setminus \{0\}$ and let
$\zeta \in \Delta$. Then set $u_t = (1+\omega t) \zeta x$ where
$t\in [0, \vartheta_\delta)$ is such that $u_t \in B$. By (\ref{d3})
and (\ref{d4}) and by Corollary \ref{G1cor} we obtain
\begin{equation*}
\rho(\Phi_{\alpha_t} (u_t),\zeta x) \leq t \|h(\zeta x)\| c (\zeta x),
\end{equation*}
which holds for some positive $c(\zeta x)$. Note that $\zeta x \in
D$ since $D$ is balanced. We conclude that, for each $r\in (\|x\|,
1)$, there exists $\vartheta_r < \vartheta_\delta$ such that
\begin{equation*}  
\| \Phi_{\alpha_t} (u_t) \| \leq r, \qquad \mathrm{whenever} \ \ t\in [0,
\vartheta_r].
\end{equation*}
Since $\Phi_{\alpha_t} \in \mathrm{Hol}(B)$,
\begin{equation}  \label{7a}
y_t (\zeta) := \Phi_{\alpha_t} (u_t) = \zeta x + \frac{t}{ 1 + \omega t} h(
y_t (\zeta)), \quad y_0 (\zeta) = \zeta x
\end{equation}
defines a holomorphic map $\Delta \ni \zeta \mapsto y_t(\zeta) \in B
\subset X$ for $t\in [0,\vartheta_r]$. At the same time, by
\eqref{d5A} the map $t \mapsto y_t(\zeta) \in B \subset X$ has a
one-sided weak derivative at $t=0^+$. For $x^* \in J(x)$, let us
consider
\begin{equation}  \label{8}
f (\zeta, t) = \frac{1}{\|x\| (1+ \omega t)}\langle y_t (\zeta), x^* \rangle
.
\end{equation}
For each $t\in [0,\vartheta_r]$, it is a holomorphic function on $\Delta$.
By (\ref{7a}), $f (\zeta, 0) = \zeta \|x\|$, and
\begin{equation*}
| f (\zeta, t)| < \frac{\|y_t (\zeta)\|}{|1+\omega t|} \leq \frac{ r }{%
|1+\omega t|}.
\end{equation*}
Thus, if $\omega \geq 0$, $|f(\zeta, t)|\leq r <1$ for all $t\in [0,
\vartheta_r]$, whereas if $\omega <0$, $|f(\zeta, t)| <1$ for sufficiently
small $t$. Hence, for such $t$, $f(\cdot, t)$ maps $\Delta$ into itself. For
each fixed $\zeta \in \Delta$, $f(\zeta, t)$ has a one-sided derivative at $%
t=0^+$. A direct computation from (\ref{7a}) and (\ref{8}) yields
\begin{equation*}
f^{\prime }_t (\zeta, 0) = [- \omega \zeta \|x\|^2 + \langle h(\zeta x), x^*
\rangle ]/ \|x\|.
\end{equation*}
Applying this in (\ref{H}) with $c = \|x\|$ and $\zeta = \bar{\zeta}
= s \in (0,1)$, we then obtain
\begin{eqnarray*}
\mathrm{Re} \langle h(s x) , x^* \rangle \leq \omega s \|x\|^2 + (1- s^2
\|x\|^2)\mathrm{Re} \langle h(0) , x^* \rangle,
\end{eqnarray*}
which holds for all $s < 1$. Thus, in the limit for $s\to 1^-$ we
get (\ref{3b}).
\end{proof}
\begin{proof}[Proof of Theorem \protect\ref{g1tm}]
By Lemma \ref{H2lm}, it is enough to show that \eqref{d5A} holds for every $%
x\in D$. By (\ref{d3}), (\ref{d4}) and Corollary \ref{G1cor}, for
sufficiently small positive $\alpha$,
\begin{equation}  \label{New}
\| (\Phi_\alpha (x) - x)/ \alpha \| \leq \frac{c_2 (x,r) }{c_1 (x,r)( 1 -
\alpha \omega)} \| \omega x - h(x)\|.
\end{equation}
If $X$ is reflexive, this estimate implies that the set
\begin{equation*}
\{(\Phi_\alpha (x) - x)/ \alpha \subset X: \alpha \in (0, \delta)\}
\end{equation*}
is relatively weakly compact. For every sequence $\{\alpha_n\}_{n\in \mathbb{%
N}}$, converging to $0$ as $n\to \infty$, $\Phi_{\alpha_n} (x) \to x$ by
Lemma \ref{D1lm}. By (\ref{4}),
\begin{equation*}  
\frac{\Phi_{\alpha_n} (x) - x}{ \alpha_n} = h(\Phi_{\alpha_n}(x)) - \omega x.
\end{equation*}
The assumed closedness (in the weak topology) of $h$, the weak
compactness just mentioned, and the fact that reflexive Banach
spaces are weakly sequentially complete then yield: for each
$\alpha_n \to 0^+$, the sequence $\{ (\Phi_{\alpha_n} (x) - x)/
\alpha_n \}_{n\in \mathbb{N}}$ converges weakly to the right-hand
side of (\ref{d5A}).

If $X$ is not reflexive but it is sequentially complete, $D$ is
balanced and $h$ is bounded on each $K(x)$, $x\in D$, we use the
following arguments. Take arbitrary $y^*\in X^*$ and consider
\begin{eqnarray}  \label{D8}
f_\alpha(\zeta)& := & \bigg{\langle } \frac{1}{\alpha}[\Phi_{\alpha} (\zeta
x) - \zeta x] , y^*\bigg{\rangle} \\[.2cm]
& = & - \zeta\omega \langle x, y^* \rangle + \langle h(\Phi_{\alpha}(\zeta
x)), y^* \rangle, \quad \zeta \in \Delta, \ \alpha \in (0,\delta].  \notag
\end{eqnarray}
For each $\alpha \in (0,\delta)$, $f_\alpha \in \mathrm{Hol} (\Delta,
\mathbb{C})$. For a fixed $\zeta \in \Delta$, the first line in (\ref{D8})
can be estimated by means of (\ref{New}), from which it follows that
\begin{equation*} 
| f_\alpha (\zeta)| \leq \frac{\|y^*\|c_2(\zeta x,r)}{c_1(\zeta x,r) c_3
(\delta \omega)} \| \zeta \omega x - h(\zeta x)\|,
\end{equation*}
where $c_3(\delta \omega)=\min\{ 1, (1- \delta \omega)\}$. Note that
$c_2(\zeta x,r)/c_1(\zeta x,r)$ can be estimated on each
 $K(x)$. Thus, in view of the assumed boundedness
of $h$ on the sets  $K(x)$, the family $\{ f_\alpha: \alpha \in (0,
\delta)\}$ is bounded uniformly on compact subsets of $\Delta$. By
Montel's theorem this family contains a sequence
$\{f_{\alpha_n}\}_{n\in \mathbb{N}}$, which converges as $n \to
\infty$, uniformly on compact subsets of $\Delta$ to some $f\in
\mathrm{Hol}(\Delta, \mathbb{C})$. Then by (\ref{D8}) we get that
the sequence $\{ \langle h(\Phi_{\alpha_n}(\zeta x)), y^*
\rangle\}_{n\in \mathbb{N}}$ converges to $f(\zeta) + \zeta \omega
\langle x, y^*\rangle$. Now we use the convergence in (\ref{gG}),
the weak sequential completeness of $X$, and the closedness of $h$
and conclude that
\begin{equation*}  
\langle h(\Phi_{\alpha_n}(\zeta x)), y^* \rangle \to \langle h(\zeta x), y^*
\rangle, \qquad n\to \infty,
\end{equation*}
for all $y^*\in X^*$.
\end{proof}
\begin{remark}
\label{Newrk} From the previous proof it follows that, for weakly
sequentially complete spaces $X$ and maps $h$ as in Theorem
\ref{g1tm}, bounded on each $K(x)$,  the map $\zeta \mapsto h(\zeta
x)$ is in $\mathrm{Hol} (\Delta , X)$ for each $x\in D$.
\end{remark}

\subsection{Dissipative and pseudo-contractive holomorphic maps}

\label{Subholo}

The following statement, see \cite[Theorem 1]{HRS} is an extension
of the well-known Lumer--Phillips
theorem, see e.g., \cite[p. 30]{Bonsal}.
\begin{proposition}
\label{Hpn} Given $h\in \mathrm{Hol}(B,X)$, it follows that $L(h) \leq 0$ if and only if,
for each $t>0$, $(I - t h)(B) \supseteq B$ and $(I - t h)^{-1} \in
\mathrm{Hol} (B)$.
\end{proposition}
An immediate corollary of Proposition \ref{Hpn} is the following statement.
\begin{proposition}
\label{1pn} If $L(h) <0$, then, for each $y \in B_r$, $r= - L(h)$, the
equation $y = h(x)$ has a unique solution $x\in B$. In particular, $h$ has a
unique null point in $B$.
\end{proposition}
\begin{lemma}
\label{1tm} Given $\omega\in \mathbb{R}$ and  $h\in \mathrm{Hol}(B,X)$, assume that
$L(h) \leq \omega$. Let also $r>0$ and $\lambda \in \mathbb{C}$ be
such that $\mathrm{Re} \lambda > \omega + r$. Then,  for each $y \in \bar{B}_r$, the equation $\lambda x -
h(x) = y$ has a unique solution in $B$.
\end{lemma}
\begin{proof}
For a given $y\in \bar{B}_r$, the null points of $g(x) := y - \lambda x +
h(x)$ solve the equation in question. For $x\in \partial B$, $x^* \in J(x)$,
and $s\in (0,1)$, we obtain
\begin{equation*}
\mathrm{Re} \langle g(s x), x^* \rangle = \mathrm{Re} \langle y, x^* \rangle
- s \mathrm{Re} \lambda + \mathrm{Re}\langle h(s x), x^* \rangle .
\end{equation*}
Thus, by (\ref{H2}),
\begin{equation*}
\limsup_{s\to 1^-} \mathrm{Re}\langle g(sx), x^* \rangle \leq \|y\| -
(\mathrm{Re} \lambda - \omega) \leq r - (\mathrm{Re} \lambda - \omega) < 0.
\end{equation*}
Hence, the result follows by Proposition \ref{1pn}.
\end{proof}
The next statement was proven in \cite[Theorem 1.5]{BES}.
\begin{proposition}
\label{BESpn} Let $h\in \mathrm{Hol}(B,X)$ be $\omega$-dissipative.
Then, for each $x \in B$, it follows that
\begin{equation*}  
\|h(x) - h(0) \| \leq \frac{\varkappa_h \|x\|}{1 - \|x\|^2} + 4 \|h(0) \|
\|x\|^2,
\end{equation*}
where the constant $\varkappa_h>0$ can be calculated explicitly. In
particular, $h$ has unit radius of boundedness.
\end{proposition}
\begin{proof}[Proof of Theorem \protect\ref{2tm}]
$(i) \Rightarrow (ii)$ is immediate. \vskip.1cm $(i) \Rightarrow (iv)
\Rightarrow (iii)$: The part related to the boundedness of $h$ follows by
Proposition \ref{BESpn}. For $\alpha\neq 0, 1/\omega $ and $z \in B$,
consider
\begin{equation}  \label{6}
(I - \alpha h)(x) = (1 - \alpha \omega )z.
\end{equation}
For $\lambda = 1/\alpha$, \eqref{6} becomes $\lambda x - h(x) = y$ with $y =
(\lambda -\omega )z$. By Lemma \ref{1tm}, whenever $1/\alpha > \omega $, the
latter has a unique solution $x\in B$, which holds for each $\alpha \in
Q_\omega$, see (\ref{5}). Thus, $\Phi_\alpha$ is defined as a self-map of $B$%
. For $\omega =0$, $\Phi_\alpha\in \mathrm{Hol}(B)$ for all $\alpha \geq 0$
by Proposition \ref{Hpn}. Similarly, $\Phi_\alpha\in \mathrm{Hol}(B)$ for
all $\omega \neq 0$ and each $\alpha \in Q_\omega$. The existence of the
first limit in (\ref{H4}) follows by Lemma \ref{D1lm} as $h$ is clearly $%
\omega$-pseudo-contractive, cf. Definition \ref{gdf}. The second limit in (%
\ref{H4}) follows by the continuity of $h$. \vskip.1cm $(iii)
\Rightarrow (i) $ follows directly from Lemma \ref{H2lm}. \vskip.1cm
$(ii) \Rightarrow (iii)$: The boundedness follows by \cite[Corollary
9]{HRS}. To prove that, for some $\delta>0$, $\Phi_\alpha \in
\mathrm{Hol}(B)$ for all $\alpha \in [0,\delta)$ we let $g:= h -
\omega I$, so that $L(g) \leq 0$. Hence, by \cite[Theorem 1]{HRS},
$(I - t g)^{-1} \in \mathrm{Hol}(B)$ for all $t>0$, \textsl{i.e.},
the map $B \ni x \mapsto y \in B$ is holomorphic, where $y$ is
defined by
\begin{equation*}  
y - t g(y) = (1+ \omega t) y - t h(y) = x,
\end{equation*}
which, for $t\in [0, 1/|\omega|)$, can be rewritten as
\begin{equation*}
y_{\alpha_t} - \alpha_t h(y_{\alpha_t} ) = (1 - \alpha_t \omega ) x.
\end{equation*}
Thus, for $0<\delta < 1/|\omega|$, it follows that $\Phi_\alpha \in
\mathrm{Hol}(B)$ for all $\alpha \in[0,\delta]$.
\end{proof}

\subsection{Nonlinear Abel averages}

\label{Abel}

Let $h\in \mathrm{Hol}(B,X)$ be such that $h(0)=0$. For every nonzero $%
\lambda \notin \sigma (h)=\sigma(h^{\prime }(0))$, the set $U$ in
the definition of $\sigma(h)$ also contains $0$, and hence one can
choose $r>0$ such that $(\lambda I - h)^{-1} (\bar{B}_r) \subset B$.
Fix these $\lambda$ and $r$. Then, for $\alpha = 1/\lambda$ and all
$\omega\in \mathbb{R}$ such that $|1 / \alpha - \omega|\leq r $, the
map
\begin{equation*}
\Phi_\alpha = \left(\frac{1}{\alpha} I - h\right)^{-1} \circ \left(\frac{1}{\alpha} -
\omega\right) I
\end{equation*}
is in $\mathrm{Hol}(B)$. Note that here we do not assume that $h$ is $\omega$%
-dissipative.

Let $\mathrm{Fix} (\Phi_\alpha):=\{x\in B: \Phi_\alpha(x)=x\}$ be
the set of fixed points of $\Phi_\alpha$. Since $\Phi_\alpha (x) =
x$ is equivalent to $( I - \alpha h) (x) = (1 - \alpha \omega) x$,
we then get
\begin{equation*}  
\mathrm{Fix} (\Phi_\alpha) = \mathrm{Null}(\omega I - h)= \{ x\in B: \omega
x - h(x) = 0\}.
\end{equation*}
Recall that a subset $R\subset B$ is called a \textsl{holomorphic
retract} if there exists a \textsl{holomorphic retraction} from $B$
on $R$, {\sl i.e.}, a holomorphic self-map $\phi$ of $B$ such that
$\phi (B)=R$ and $\phi(z)=z$ for all $z\in R$. If $R$ is a
holomorphic retract of $B$ then, in particular, it is a non-singular
closed submanifold of $B$, and it is also totally geodesic with
respect to the hyperbolic metric of $B$.

Combining classical results of Koliha \cite[Theorem 0]{Koliha} and
Vesentini's \cite[Theorem 1]{Ves}, one can obtain the following
characterization of the power convergence of holomorphic maps.
\begin{proposition}
\label{kolihapn}Let $\Psi \in \mathrm{Hol}(B)$ be such that $0\in
\mathrm{Fix}(\Psi )$. Then the following statements are equivalent.
\begin{itemize}
\item[{\it (a)}]  The sequence of iterates $\{\Psi ^{n}\}_{n\in \mathbb{N}}$ converges, in the operator
norm topology,
uniformly on closed subsets of $B$, to a holomorphic retraction of $B$ onto $\mathrm{Fix} (\Psi)$.
\item[{\it (b)}] The sequence $\{(\Psi'(0))^{n}\}_{n\in \mathbb{N}}$ is convergent in the operator
norm topology.
\item[{\it (c)}] $\sigma (\Psi'(0))\subset \Delta \cup \{1\}$ and $\varsigma =1$ is at most
a simple pole of the resolvent of $\Psi'(0)$.
\end{itemize}
\end{proposition}
\begin{proof}[Proof of Theorem \protect\ref{3tm}]
Note that $(\alpha I-h^{\prime }(0))$ is invertible as $1/\alpha \notin \sigma (h)=\sigma(h^{\prime
}(0))$. Since $\Phi_\alpha\in\mathrm{Hol} (B)$, we can compute its Fr{\'e}chet derivative $\Phi_\alpha^{\prime }(x)$, $x\in B$.
By the chain rule we
get from (\ref{4})
\begin{equation*}
\Phi_\alpha^{\prime }(x) - \alpha h^{\prime }( \Phi_\alpha (x))
\Phi_\alpha^{\prime }(x) = (1-\alpha \omega) I.
\end{equation*}
Next, we have $\Phi_\alpha (0) =0$ as $h(0) = 0$, and hence
\begin{equation}  \label{15}
\Phi_\alpha ^{\prime }(0) = (1 -\alpha\omega ) (I - \alpha h^{\prime}(0))^{-1}.
\end{equation}
Set
\begin{equation*}  
\psi (\zeta) = \frac{1- \alpha \omega}{1 - \alpha \zeta}, \qquad \zeta \in
\mathbb{C}.
\end{equation*}
This function is holomorphic on $\Omega(\alpha, \omega)$, and $%
\psi(\Omega(\alpha, \omega))= \Delta\cup\{1\}$. On the other hand, (\ref{15}%
) and the spectral mapping theorem imply
\begin{equation}  \label{17}
\sigma \left( \Phi_\alpha ^{\prime }(0) \right) = \sigma \left( \psi
(h^{\prime }(0))\right) = \psi\left(\sigma (h^{\prime }(0))\right).
\end{equation}
Suppose now that $\sigma (h)\subset \Omega(\alpha, \omega)$. Then, by (\ref%
{17}) it follows that
\begin{equation*}  
\sigma \left( \Phi_\alpha ^{\prime }(0) \right) \subset \Delta\cup\{1\}.
\end{equation*}
Taking into account that $1/\alpha \notin \sigma(h^{\prime }(0))$, and hence
$\mathrm{Im} ( I - \alpha h^{\prime }(0)) = X$, direct computations yield
\begin{eqnarray}  \label{19}
\mathrm{Ker} (I - \Phi_\alpha ^{\prime }(0)) & = & \mathrm{Ker} (\omega I -
h^{\prime }(0)), \\[.2cm]
\mathrm{Im} (I - \Phi_\alpha ^{\prime }(0)) & = & \mathrm{Im} (\omega I -
h^{\prime }(0)).  \notag
\end{eqnarray}
By (\ref{19}) and (\ref{13}), we then get
\begin{equation*}  
\mathrm{Ker} (I - \Phi_\alpha ^{\prime }(0)) \oplus \mathrm{Im} (I -
\Phi_\alpha ^{\prime }(0)) = X,
\end{equation*}
which means that $1$ is at most a simple pole of the resolvent of
$\Phi_\alpha ^{\prime }(0)$. Then by
Proposition \ref{kolihapn}, the sequence $\{\Phi _{\alpha }^{n}\}_{n\in \mathbb{N}}$
converges, uniformly on closed subsets of $B$, to a holomorphic retraction
$\phi_\alpha :B\rightarrow \mathrm{Fix}\Phi _{\alpha }=\mathrm{Null}(\omega I-h)$.
The converse statement  follows directly by  Proposition \ref{kolihapn}.
\end{proof}
\begin{remark}
As follows from Theorem \ref{3tm}, a \textit{necessary} condition for the
sequence $\{\Phi^n_\alpha\}_{n\in \mathbb{N}}$ to converge is that $\mathrm{%
Null}(\omega I - h)$ is a holomorphic retract of $B$.
\end{remark}

\section{Examples}

\label{examples}

\subsection{Dissipative maps}

\label{Essec}

As a typical example of a map $h:D\subset B \to X$ described by Theorem \ref%
{g1tm} we consider
\begin{equation}  \label{E1}
h = T + g,
\end{equation}
where $T:\mathcal{D}(T) \subset X \to X$ is a closed densely defined linear
operator with a nonempty resolvent set and $g\in \mathrm{Hol} (B,X)$. In
this case, $D= \mathcal{D}(T) \cap B$. As we pointed out in Section \ref{S1}, $T$ is also closed in the weak topology. Hence so is $h$. For
each $x\in B$, $h$ is bounded on compact subsets of $K(x)$. Moreover, by Remark \ref{Newrk},
the map $\zeta \mapsto h(\zeta x)$ is in $\mathrm{Hol}(\Delta, X)$.
\begin{proposition}
\label{E1as} Let $T:\mathcal{D}(T) \subset X
\to X$ be  closed and densely defined, and such that $\mathrm{Re}\langle Tx, x^ * \rangle \leq 0$
for each $x\in \mathcal{D}(T)$ and $x^*\in J(x)$. Assume
also that $g\in \mathrm{Hol}(B , B_\omega)$ for some $\omega >0$.
Then $h$ defined in (\ref{E1}) is $\omega$-dissipative.
\end{proposition}
\begin{proof}
The assumed properties of $T$ imply that its Abel average $A_\alpha$, $\alpha \geq 0$, defined in (\ref{i3}) exists and satisfies (\ref{i7}). For $\alpha>0$, the Abel average of $h$ (\ref{4}) maps $x\in B$ to $y\in X$ given
by the unique solution of the equation
\begin{equation*}
y - \alpha T y - \alpha g (y) = (1- \alpha \omega ) x.
\end{equation*}
This can be rewritten as
\begin{equation*} 
y = \alpha (A_\alpha \circ g) (y) + (1- \alpha \omega) A_\alpha
x.
\end{equation*}
By Proposition \ref{1pn}, the map
\begin{equation*}
y \mapsto \alpha (A_\alpha \circ g) (y) + (1- \alpha \omega)
A_\alpha x - y
\end{equation*}
has a unique null point in $B$. Thus, the map $x \mapsto y$ is in $\mathrm{%
Hol} (B)$. The proof now follows by Theorem \ref{g1tm}.
\end{proof}
A concrete example of $h$ as in Proposition \ref{E1as} is given by the following integro-differential map, which appears in
nonlinear and nonlocal evolution equations of the Fisher-KPP type
\cite{B}. Here $X$ is a complex Hilbert space $L^2 (\mathbb{R})$.
Let $a:\mathbb{R}\times \mathbb{R}
\to (0,+\infty)$ be symmetric and such that the operator
\begin{equation*}
L^2 (\mathbb{R} )\ni x \mapsto \int_{\mathbb{R}} a( \cdot, s ) x(s) ds,
\end{equation*}
maps $L^2 (\mathbb{R} )$ into $L^\infty (\mathbb{R} )$, and
\begin{equation*}
\bigg{\|}\int_{\mathbb{R}} a( \cdot, s ) x(s) ds%
\bigg{\|_{L^\infty
(\mathbb{R} )}} \leq a \| x\|_{L^2 (\mathbb{R} )},
\end{equation*}
for some $a>0$. The integro-differential map $h=T + g$, where
\begin{equation*}
T = \frac{d^2}{dt^2} , \qquad g(x) = bx(\cdot)\bigg{[}1- \int_{\mathbb{R}} a
(\cdot, s) x(s) d s\bigg{]}, \quad \ b>0,
\end{equation*}
is $1$-dissipative for $b< 1 / (1 + a)$.  For $\mathcal{D}(T)$, one can take
the Sobolev space $W^{2,2} (\mathbb{R})$, see, {\sl e.g.},
\cite[Chapters 6 and 7]{LL}.

\subsection{Nonlinear Abel averages}

\label{dep} For a linear operator with Abel average $A_\alpha$, $\alpha\in (0,1)$,
the limit of the sequence $\{A_\alpha^n\}_{n\in \mathbb{N}}$, if it exists, is
one and the same for all $\alpha$, see Assertion \ref{ias2}.  In the nonlinear case, this is no longer true. This is
related to the non-uniqueness of holomorphic retractions for holomorphic
retracts of $B$.
In \cite[Section 3]{BPT} it was proved that any one-dimensional retract
of a bounded convex domain in $\mathbb{C}^n$ with smooth boundary admits a
unique holomorphic retraction whose fibers are affine. However, in general, there
can also be non-affine retractions. Using Abel averages for nonlinear
holomorphic maps, one can construct such non-affine holomorphic retractions, as it is shown
in the following example shows.

Denote $\mathbb{B}^2:=\{z \in \mathbb{C}^2: \|z\|<1\}$,  where $\| \cdot \|$ is the
standard Euclidean norm, and set $h(z) = h(\xi,\eta)=(\lambda \xi+\epsilon \eta^2, 0)$,
$0<\epsilon<1$. Note that $\sigma(h)=\sigma(h^{\prime }(0))=\{0,\lambda\}$.
Set $\omega=\lambda$. Then $\sigma(h)\subset \Omega(\alpha,\lambda)$
for $|1-\alpha\lambda|<1$. A direct computation shows that
\begin{equation}
\label{F1}
\Phi_\alpha(\xi,\eta)=(\xi+\alpha \epsilon (1-\alpha\lambda)\eta^2,
(1-\alpha\lambda)\eta).
\end{equation}
It is easy to check that $\Phi_\alpha$ is a
holomorphic self-map of $\mathbb{B}^2$ for small enough $\alpha$.
Since \eqref{13} is always satisfied in finite dimensional Banach
spaces, Theorem \ref{3tm} applies and $\{\Phi_\alpha^n\}$ converges to a
holomorphic retraction $\phi_\alpha$ of $\mathbb{B}^2$ onto
$\mathrm{Null}(\lambda I-h)=\{(\xi,0): \xi\in \Delta\}$.
By (\ref{F1}), for $n\in \mathbb{N}$, we then get
\begin{equation*}
\Phi_\alpha^n(\xi,\eta)=(\xi+\alpha\epsilon (1-\alpha\lambda) \eta^2 \sum_{j=0}^{n-1}
(1-\alpha\lambda)^{2j}, (1-\alpha\lambda)^n \eta),
\end{equation*}
which yields
\begin{equation*}
\phi_\alpha(\xi,\eta)=\lim_{n\to \infty} \Phi_\alpha^n(\xi,\eta)=(\xi+\epsilon\frac{
1-\alpha\lambda}{2-\alpha\lambda} \eta^2, 0).
\end{equation*}
In particular, $\phi_\alpha$
depends on $\alpha$.

It is also interesting to note that $\phi_\alpha$ can be extended to all $\alpha \in \mathbb{R}$ such that $|1-\alpha\lambda|<1$.
Although  the map
$\Phi_\alpha$ may no longer be a self-map of $\mathbb{B}^2$ for some $\alpha$. For instance, take
$\lambda=1$, $\epsilon=1/2$, and  $\alpha\in (0,2)$. But,
$\lim_{|\eta|\to 1}\|\Phi_\alpha(0,\eta)\|=(1+\frac{\alpha^2}{4})(1-\alpha)^2$ and,
for $\alpha$ close to $2$, this number is bigger than $1$. In this  case, however,
$\phi_\alpha$ is not a self-map if $\mathbb{B}^2$, so that it is not a
holomorphic retraction of $\mathbb{B}^2$ onto $\mathrm{Null}(\lambda I-h)$.

For $\omega=0$, it follows that $\sigma(h)\subset \Omega(\alpha,0)$ for all $%
\alpha\in \mathbb{R}$ such that $|1-\lambda\alpha|>1$. In this case,
\[
\Phi_\alpha(\xi,\eta)=\left(\frac{\xi+\alpha \epsilon \eta^2}{1-\alpha\lambda}, \eta\right).
\]
It is
easy to check that, for any $\alpha$ such that $|1-\lambda\alpha|>1$, the
points $\Phi_\alpha(-\frac{\epsilon}{\lambda} \eta^2, \eta)$ are not in $\mathbb{B}%
^2$ for $|\eta|\to 1$. This is not surprising, because otherwise $%
\{\Phi_\alpha^n\}$ would converge to a holomorphic retraction of $\mathbb{B}%
^2$ onto $\mathrm{Null}(h)=\{(\xi,\eta)\in \mathbb{B}^2: \lambda \xi=-\epsilon
\eta^2\} $, which would imply that $\mathrm{Null}(h)$ be a one-dimensional
holomorphic retract of $\mathbb{B}^2$ (so-called \textit{complex geodesic
of $\mathbb{B}^2$}), while one-dimensional holomorphic retracts of $\mathbb{B%
}^2$ are known to be just the intersection of affine complex lines with $%
\mathbb{B}^2$.

\section{Open question}

Let $D$ be a balanced dense subset of $B$, and let $h$ be a closed
(weakly closed) map on $D$ with values in $X$. Assume that $h$ is
$0$-dissipative and for  some $\alpha _{0}>0$ its Abel average
exists, is holomorphic,
and maps $B$ into itself. Do the Abel averages exist for all positive $%
\alpha $? \vskip.2cm \noindent \textbf{Acknowledgement:} The authors
cordially thank Simeon Reich for valuable suggestions and comments.
Yuri Kozitsky is grateful to ORT Braude College for warm hospitality
extended to him during his work on the paper.

\end{document}